\documentclass[a4paper,12pt]{amsart}
\usepackage{amsmath,amssymb,colordvi}
\usepackage{graphicx}

\newtheorem{theorem}{Theorem}

\newtheorem{example}[theorem]{\it Example}
\newtheorem{proposition}[theorem]{Proposition}
\newtheorem{definition}[theorem]{Definition}

\newtheorem{remark}[theorem]{\it Remark}

\font\tenBb=msbm10 \font\sevenBb=msbm7 \font\fiveBb=msbm5
\newfam\Bbfam \textfont\Bbfam=\tenBb \scriptfont\Bbfam=\sevenBb
\scriptscriptfont\Bbfam=\fiveBb

\def\Q{{\Bb Q}}

\def\Q{{\mathcal Q}}

\begin{document}
\title{A functional treatment of asymmetric copulas}

\author{Ahmed Sani }
  \address{Universit\'e Ibn Zohr, Facult\'e des Sciences, D\'epartement de Math\'ematiques, Agadir, Maroc}
 \email{ahmedsani82@gmail.com }

 \author{Loubna Karbil}
  \address{Universit\'e Hassan I, Facult\'e des Sciences, D\'epartement de Math\'ematiques, Casablanca, Maroc}
  \email{l.karbil@gmail.com}


\footnote{\today}

\subjclass[2010]{Primary   62 Hxx;62 Exx Secondary 46 Axx  }


\begin{abstract}
The concept of asymmetric copulas is revisited and is made more precise. We give a rigorous topological argument for opportunity to define asymmetry measures defined recently by K.F Siburg \cite{Si16} through exhibiting at least three ordered classes  of copulas according to a suitable equivalence relation. We define a process of ordering  subcopulas  which makes clearer the  degree of asymmetry. As illustration, we treat the asymmetric Cobb-Douglas utility model.

\end{abstract}

\maketitle

\section*{Introduction}
\setcounter{theorem}{0}
 \setcounter{equation}{0}

Recently copulas are  becoming an emerging concept in probability and statistic models. The copula permits in particular to model dependence structure between random variables. There are various and important applications of copulas in finance, actuarial and risk management. For a given data set, the practical and real problem is to determine a copula (or copulas) which fits for describing the dependence structure.  Our summarized note on copulas is based on the edifying paper of A. Sklar and the theorem that bears his name\cite{Sk59}  and the unavoidable book of R. Nelson \cite{Ne06}. The copula is a notion which allows a  normalization of  random variables law. Assume for example that one has a random vector $(X,Y)$ where $X$ and $Y$ act on different probability spaces and eventually with different laws described by their distribution functions $F^X$ and $F^Y$ respectively. Let $H$ denote the joint distribution of $X$ and $Y.$ In order to highlight properties of the couple $(X,Y)$ concerning dependence, regression or symmetry it will be convenient  to  project $X$ and $Y$ in a common space where their projections share the same law and where arithmetic operations are rather possible. It is a classical result that $F^X(X)$ and $ F^Y(Y)$ give a nice answer mainly when the margin distributions $F^X$ and $F^Y$ are continuous. In fact, $(F^X(X), F^Y(Y))$ has a uniform distribution $(U,V)$ on $\mathbf I^2=[0,1]\times [O,1]$ and a copula of the $(X,Y)$ is exactly such as  $C_{X,Y}(u,v)=H(x,y)$ as soon as $(u,v)=(F^X(x),F^Y(v))$. This technique of normalization consisting on researching of common projection space becomes recently a sufficient tool in many fields where heterogeneous data are used. We refer for example to \cite[III,A]{Ta16} concerning the projection of heterogeneous multimedia data into a common low-dimensional
Hamming space; and to \cite{Di05} for similar procedure using \emph{Krylov method} to reduce power systems by projecting data in smaller finite dimensional space.  \\
In order to give more autonomy to this work, we recall the essential results on copulas. Without loss of generality,  we restrict ourself to the bivariate case. To this end we consider the unit closed square $I^2=[0,1]\times [0,1].$ 
\begin{definition}\label{def1}
A copula $C$ is a function on $I^2$ into $I=[0,1]$ which satisfies the following conditions for all $(u,v)\in I^2$:
\begin{enumerate}
	\item $C(0,v)=C(u,0)=0.$
	\item $C(1,v)=v$ and $C(u,1)=u.$
	\item the 2-increasing property: $C(u_2,v_2)-C(u_2,v_1)-C(u_1,v_2)+C(u_1,v_1)\ge0.$
\end{enumerate}

\end{definition}
When we deal with copulas, it is almost impossible to overcome Sklar's theorem which is the first bridge between marginal statistical distributions and their joint multivariate distribution. Let us give a restitution of this famous theorem in our context of bivariate case

\begin{theorem}[Sklar's theorem]
Let $H$ be two-dimensional distribution function on a probability space $(\Omega,p)$  with marginal distribution functions $F$ and $G.$ Then there exists
a copula C such that
\begin{equation}\label{eq1}
H(x,y) = C(F(x),G(y)).
\end{equation}

If  $F$ and $G$ are  continuous then the copula $ C$ is unique.
Conversely, for any copula $C$ and any distribution functions $F$ and $G$ the
 equation (\ref{eq1}) defines a two-dimensional distribution function $H$ with
marginal distribution functions $F$ and $G$.
\end{theorem}
 
Mathematically speaking, it is not difficult to establish that the three conditions above (in definition  \ref{def1}) imply that a copula enjoys certain analytical properties such continuity and monotonicity  on each argument. Moreover using  condition (3) one may prove easily that a copula as defined in definition (\ref{def1})  is Lipschitz-continuous and then almost everywhere differentiable. By a classical integration theorem, it is so possible to rebuild the copula from one of its partial derivative. Algebraically speaking, the well known Fr\'echet-Hoeffding result states  very important and canonical estimations for any given copula as formulated in the following proposition
\begin{proposition}
The functions $C_l$ and $C_u$ defined for all $(u,v)\in I^2$ by setting $C_l(u,v)=\min(u,v)$ and $C_u(u,v))=\max(u+v-1,0)$ are copulas. Furthermore, for every copula $C,$ one has
$$
\forall (u,v)\in I^2 \quad \quad C_l(u,v)\le C(u,v)\le C_u(u,v).
$$ 
\end{proposition}
 In the litterature, $C_l$ and $C_u$ are said respectively  \emph{the  Fr\'echet-Hoeffding lower bound} and the  \emph{the Fr\'echet-Hoeffding upper bound}.\\
 It is worth to mention that these bounds depend on the natural order defined on the set of all copulas as follows
 $$
 C_1\le C_2 \ \Longleftrightarrow  \ \forall (u,v)\in I^2:\ C_1(u,v)\le C_2(u,v).
 $$ 
This natural order is far from being total (it is just a preorder). For example,  according to \cite[Example 2.18]{Ne06}, the copulas $\frac{C_l+C_u}{2}$ and $\Pi :(u,v)\mapsto \pi(u,v)=uv$ are not comparable. Other orders are possible. For example, in \cite{De10},
a general concept of a regression dependence order (RDO, as denoted by the authors) was introduced. Recently, to characterize the symmetry degree of a given copula compared to another K.F. Siburg et al. \cite{Si16} defined a partial order on copulas set by introducing the difference between a copula $C$ and its transpose $C^{T}$ defined by

$$
\forall (u,v)\in I^2 \ \ C^{\top}(u,v)=C(v,u).
$$
Let us recall the order adopted there (in\cite{Si16}). We say that a copula $C_1$ is \emph{more symmetric} than another one $C_2$ and we write $C_1 \preceq C_2$ if and only if, roughly speaking,  the difference between $C_1$ and its transpose is smaller than the difference between $C_2$
and $C_2^{\top}.$ To be more precise, we state:

\begin{equation}\label{eq2}
	C_1 \preceq C_2   \Longleftrightarrow 
	 \forall (u,v)I^2:|C_1(u,v)-C_1^{\top}(u,v)|\le  |C_2(u,v)-C_2^{\top}(u,v)|.
\end{equation}

For the partial order "$\preceq$" the lower bound is manifestly the product copula, said also \emph{independence copula} $P$ given by $\forall(u,v)\in I^2\ P(u,v)=uv.$ Moreover $P$ is the smallest element of the partially ordered set $(\mathcal {CP} ,\preceq))$ where $\mathcal {CP}$ denotes the set of all copulas. Unfortunately as proved in \cite{KM06}, this order has not an upper bound and so a fortiori has not a greatest element nor an upper bound.\\
The main result of this paper is to establish that the order $"\preceq"$ is not trivial in the following sense: the classes of copulas with the same degree of symmetry are not reduced to one. Thus for a given measure of symmetry (or asymmetry), it will be convenient to search the \emph{most symmetric} among a  parameterized  copulas. Topological arguments are given  to prove our purposes.\\

In the first section, we recall and reshuffle the most important results on asymmetric copulas. As the mathematical ingredients are needed to make clear and precise, we  introduce,  at the second section, all mathematical tools which will be used in the remaining of this paper. The third section is devoted to the main results and their proofs. At last we give a way to define a new measure of asymmetry baptized "\emph{local asymmetry measure.}"

\section{Asymmetric copulas}

A large and most well known class of copulas is the Archimedean ones. A copula $C$ is said Archimedean if there exists a decreasing  and convex function $\phi:\ [0,1]\mapsto \mathbb R^+$ with $\phi(1)=0$ such that for all $(u,v)\in I^2$  we have $C(u,v)=\phi^{[-1]}(\phi(u)+\phi(v)).$ The function $\phi$ is said \emph{ a generator} of $C.$ Here  $\phi^{[-1]}$ denotes the general inverse of $\phi.$ A lot of researches were axed on generalizing this class of copulas. For example, in
\cite{Ay14} the authors defined a general copula as been every function $C_g$ that we may write $C_g(u,v)=\phi^{[-1]}(\phi(\max(u,v))+\psi(\min(u,v))$ where $\phi$ is a strictly decreasing  continuous function, $\psi$ is assumed to be  just decreasing continuous function and $\psi-\phi$ is increasing. Further steps  have been taken in this direction. Among others, different ratios of correlation and dependence like $\tau$ of Kendall and $\rho$ of Spearman
were calculated in terms of the generator $\phi$ and its cogenerator $\psi$ of generalized copula.\\
unfortunately all of these copulas are symmetric (i.e $C(u,v)=C(v,u),\ \forall (u,v)\in I^2.$). In practice, this kind of copulas does not suffice to model 
phenomena which do not present asymmetric data. Awareness of the interest of this type of copulas began with the publication of the  Klement and Mesiar paper \cite{KM06}. Just a year after, appeared the article of Nelson \cite{Ne07} where the concept of asymmetric copula is dealt with, in some way, to the equivalent notion of  non exchangeability of two random variables. As a consequence, many tests of parametric and nonparametric symmetry tests are implemented. As a revealing example, we refer to \cite{Si16} and  the important paper of A. Erdely and  J. M. Gonz\`{a}lez-Barrios \cite{EG10}  and in another and close context \cite{MR14}. We mention also the attempt of dependence coefficient \emph{symmetrization} as initiated by Cifarelli et al.\cite{Ci96}.\\

 Recently Siburg et al. \cite[Definition 2.2]{Si16} define via (\ref{eq2}) an order on the set of all copulas. This is in fact the starting point of some important remarks and ideas that  we will mention or which are already mentioned in \cite{Si16}.\\
First, the order $\preceq$ does not distinguish between elements of a large class of copulas. To see this, it suffices to remember that it is possible, for suitable copulas, to have simultaneously $C_1\preceq C_2$ and $C_2\preceq C_1$ but they do not turn equal. This is clear for every Archimedean copulas or generally for every symmetric ones. For more general case, one may choose $C_2=C_1^{\top}$ as given in \cite[Remark 2.5]{Si16}. A more serious drawback is the incompleteness problem: A large class of copulas are not comparable. To overcome these inconsistencies, the  asymmetry measures are introduced as a complement of the order $\preceq.$ A large catalog of asymmetry and symmetry measures are proposed in many works mainly, among others, \cite{PP12}. So it will be interesting to know how many classes of copulas that have not the same measure. We will take advantage of the growth of measures and some elementary topological results to prove that there are at least three classes of copulas without exhibiting them explicitly when we deal with quotient topological space.  All these ingredients will be developed in the next section. 
In the fourth section, we define a new and general measure of asymmetry based on Zygmond-Calderon theorem \cite{Kr02} for general copulas and give a bridge with the known Arrow's impossibility theorem " \emph{Social choice and individual values}" published in 1951. More information on this important  result in political economy and economic science may be found in \cite{Bl69} or in the last research on this topic \cite{Yu15}. The generality of our measure consists on its ability to compare subcopulas (see \cite{Ne06} and/or section 3 below) instead of copulas.

\section{Mathematical ingredients}
Topological results presented in this section are known and may be consulted in any elementary course of topology and functional analysis. We suggest as a typical one the  famous Choquet's book \cite{Ch66} and the  irreplaceable monograph of H. Brezis \cite{Br83} or the recent and English version \cite{Br11}.\\
Let $X$ be a toplogical space. An equivalence   denoted  $\sim$ on $X$ is a relation which responds to three well known properties
\begin{enumerate}
	\item Reflexivity : $x\sim x , \forall x\in X$.
	\item Symmetry: $x\sim y \iff y\sim x.$
	\item Transitivity:  $x\sim y\  \mbox{and}\  y\sim z \  \implies x\sim z.$
\end{enumerate}

For a given $x$ in $X$ the class (or orbit)  $\dot{x}$ of $x$ is the set of all elements $y$ which are related with $x$ (i.e $\dot{x}=\{y\in X,\ x\sim y.\}$) The equivalence classes give a partition of $X$ and the mapping $x\in X\mapsto \dot{x}\in X/ \sim$  is said the \emph{canonical projection} of $X$ into the set  $X/ \sim$ of a all orbits. It is a classical question to ask what adequate topology one may  define on $X/ \sim.$ The following proposition gives an answer 
\begin{proposition}
Let $X$ be a topological space and $\sim$ an equivalence on $X.$ Consider the canonical projection $p: X\rightarrow X/\sim.$ The quotient topology is the family 
\begin{equation}\label{equat1}
\tau=\{U\subset X/ \sim,\ p^{-1}(U) \ \mbox{is open set in }\ X \}.
\end{equation}
\end{proposition}

The topology on $X/ \sim$ given by (\ref{equat1})  is the optimal one which makes the canonical projection continuous. Henceforward, $X/ \sim$ will be endowed with this canonical topology.\\
A most important property of a topology on $X$ is the \emph{separation} which means that if $ x$ and $y$ are different elements of $X,$ one can always find two respective neighbors $O_x$ and $O_y$ of $x$ and $y$ such that $O_x\cap O_y=\emptyset.$ Unfortunately, the quotient topology is rather pathologic because the separation property is not always warranted. We mention  without developing it the classical and illustrative  example:
$$
X=\mathbb R \  \mbox{and} \  x\sim y \iff x-y\in \mathbb Q.
$$
One proves easily that the quotient space $\mathbb R/ \sim$ is not separate. As a consequence, among quotient topological spaces it will be interesting to characterize those which are separate. To this end, we give a short answer that we will use as a recipe to establish our first  main result on copulas in the next section 
\begin{theorem}\label{thmseparate}
Let $\sim$ be an equivalence relation on a topological space $X.$
Assume that one of the following conditions is satisfied
\begin{description}
	\item [a] The relation $\sim$ is open.
	\item [b] The space $X$ is compact and separate.
\end{description}
Then the quotient space $X/ \sim$ is separate.
\end{theorem}
The relation $\sim $ is said open if the canonical projection is so. We will not give more details because we will just use  the second item in theorem (\ref{thmseparate}). This latter condition implies that the compactness of the set of all copulas is needed. In functional analysis, the natural tool that we have at disposal is 
\begin{theorem}[Ascoli-Arzela]\label{Ascoli-Arzela}
Let $ X$ be a compact Hausdorff space and   $\mathbf F\subset C(X, \mathbb R) $. Then $\mathbf F$ is compact  if and only if it fulfills simultaneously the two conditions
\begin{description}
	
\item[i] The set $\mathbf F$ is equicontinuous, i.e for  every $x\in X $ and every $\varepsilon > 0$, $x$ has a neighborhood $ O_x$ such that

$$  \forall y\in O_{x},\forall f\in \mathbf {F} :\ |f(y)-f(x)|<\varepsilon $$

\item[ii]The set $\mathbf F$ is pointwise relatively compact and closed which means that for each $x\in $, the set $F_x = \{ f (x) :  f  \in \mathbf F\}$ is relatively compact in $\mathbb R$.

\end{description}
\end{theorem}
More general versions of this important theorem exist. In particular, it is possible to replace $\mathbb R$ in theorem (\ref{Ascoli-Arzela}) with any metric space $Y$. But for our purposes, the real case ($Y=\mathbb R$) is largely enough. \\
Now, to make easier the introduction of our new asymmetry measure, the key will be a fundamental result in integration theory. It is a kind of suitable decomposition of element of $L^1([0,1])=\{ f: I \rightarrow \mathbb R^+ \ \mbox{measurable such that }\ \int_{0}^{1}|f(x)|dx <\infty  \}.$  For such functions, the values $f(x)$ may be greater as one may want, but Calderon-Zygmund theorem ensures  that,  \emph{on average} or \emph{on expectation}, the mean values of $f$ in some disjointed subsets can be controlled. On the complement of those subsets, the function $f$ is entirely dominated. For more precise, we have (see \cite{Kr02})
 
\begin{theorem}[Calderon-Zygmund Lemma]\label{Cal-Zyg}
Consider $f \in L^1(I^2)$ and 
$t > 0.$ There exist at most countably many closed intervals $Q_i = [a_i, b_i] , i \in \mathbb N$ with
disjoint interiors such that
\begin{equation}
	\frac{1}{b_i-a_i}\int_{I^2} |f(x,y)|dxdy=t
\end{equation}

and
\begin{equation}
\|f \mathbf 1_{\{I^2\setminus {\cup_i Q_i }\}}\|_{\infty} <t.
\end{equation}
The intervals $(Q_i)_i$ are such that $Vol(\underset{i\ge 1}{\cup Q_i)}=\sum_{i=1}^{\infty} b_i-a_i \le \frac{1}{t}\|f\|_{L^1}.$

\end{theorem}
In the literature, this result is known to be the Calderon-Zygmund Lemma. The theorem that bears this name states that every function $f\in L^1$ may be decomposed as follows: $f=g+b$ where $g$ reveals "good function" and $b$ connotes "bad" one in the following sense

\begin{proposition}[Calderon-Zygmund Theorem]\label{Propos-Zygm}
Under the same assumptions 	of theorem (\ref{Cal-Zyg}), the function $f\in L^2(I^2)$ may be written $f=g+b$ where $g$ and $b$ are two functions on $I^2$ satisfying
\begin{itemize}
	\item $g(x)\le 2t.$
	\item for every $1\le p\le \infty,$ one has $\|g\|_{L^p}\le K_p\|f\|_{L^1}^{\frac{1}{p}}$ for some constant $K_p$ depending uniquely on $p.$
	\item $\int_{Q_i}b(x,y)dxdy=0.$
	\item $\|b\|_{L^1}\le 3 \|f\|_{L^1}.$
\end{itemize}
	
\end{proposition}

The  two results above will allow us to define a suitable asymmetry measure when we deal with subcopulas. Before developing the sequel, a remark  on functional analysis results exposed above deserves to be mentioned: all formulations are adapted to the context and to the needs. In \cite{Kr02}, one find more general statements and many applications.

\section{classes of copulas}

In the sequel we will denote by $\mathcal {CP},$ the set of all copulas. We state first that according to definition (\ref{def1}), $\mathcal {CP}$ is a convex subset of the Banach space $(C(I^2),\|.\|_{\infty})$ of all continuous functions on $I^2$ endowed with its natural uniform convergence norm $\|f\|_{\infty}=\max_{(x,y)\in I^2}|f(x,y)|.$ This is an immediate consequence of the proposition below. Convexity of $\mathcal {CP}$ is easy to establish since conditions (1), (2) and (3) in definition (\ref{def1}) are verified for every convex combination $tC_1+(1-t)C_2, \ 0\le t\le 1$ of two copulas $C_1$ and $C_2.$
\begin{proposition}
	Let $C$ be a copula as given by definition (\ref{def1}). For all $ (u_1, u_2)$, $ (v_1, v_2)$ in $I^2$ we have
	
	\begin{equation}\label{lipsh_cont}
	|C(u_2, v_2) - C(u_1, v_1) | \le |u_2-u_1| +|v_2-v_1|
	\end{equation}
	
\end{proposition}
The estimation (\ref{lipsh_cont}) implies that the mapping $(x,y)\mapsto C(x,y)$ is Lipschitz function and thus continuous on $I^2.$ Moreover, as already mentioned at the introduction, all partial derivative of this mapping exist almost everywhere (exactly in Lebesgue point)  and in particular  for almost all $(x,y)\in I^2$, we have $C(x,y)=\int_{0}^{x}\frac{\partial C(s,y)}{\partial s}ds.$ We refer for such results to \cite{Ev98}. \\
We are now able to state the important results concerning the compactness of  $\mathcal {CP}$
\begin{theorem}\label{CP-compact}
	The set $\mathcal {CP}$ of all copulas on $I^2$ is convex and compact subset of  $C(I^2).$
\end{theorem}
\begin{proof}
	
	The convexity is already discussed above and according to theorem (\ref{Ascoli-Arzela}) it suffices to verify that  $\mathcal {CP}$ is equicontinuous and relatively compact. The first property is easy to check from (\ref{lipsh_cont}) since for a given point $(x_0,y_0)\in I^2\in ,$ and all $\varepsilon>0$ one has
	$$
	\forall(x,y)\in I^2\ \mbox{and}\ \forall C\in \mathcal {CP} \  \|(x,y)-(x_0,y_0)\|_1< \varepsilon \implies |C(x,y)-C(x_0,y_0)|<\varepsilon.
	$$
The pointwise relative compactness is obvious because for all $f\in  \mathcal {CP} $  the set $\overline{{\{f(x,y),\ (x,y)\in I^2\}}}$	is compact as a closed subset of the compact set $I.$ 

\end{proof}

\begin{remark}
	
	The proof of theorem (\ref{CP-compact})  is made easy thanks to boundedness of all copulas which take their values in $I$(they are joint distribution in some sense). For general subsets of equicontinuous functions on $I^2,$ the pointwise relative compactness is more subtle.	
\end{remark}	

Consider now the topological space $  \mathcal {CP}$ equipped with the topology arising from the $C(I^2)$ norm. Indeed it is a metric space with the natural distance associated with the norm $\|.\|_{\infty}.$ It is then obviously separate.\\
On the separate topological space $ \mathcal {CP},$ we consider the binary relation given by 
\begin{equation}
\forall (C_1,C_2)\in \mathcal {CP}^2:\ C_1\sim C_2 \iff |C_1-C_1^{\top}|=|C_2-C_2^{\top}|.
\end{equation}

Without  any difficulty one  checks that $\sim$ is an equivalence on $\mathcal{CP}.$ A natural following step is to investigate the quotient set $\mathcal{CP}/\sim.$ The objective is to establish that this latter set is not trivial. This will give sense first for the order defined in \cite{Si16} and second for the opportunity to search and develop  asymmetry measures. The following theorem gives an affirmative answer
\begin{theorem}
The quotient $\mathcal {CP}/\sim$ contains more than three elements and thus is not trivial.
\end{theorem}
\begin{proof}
As it will be showed below, it suffices to prove that the quotient space $\mathcal {CP}/\sim$ equipped with the quotient topology is separate. In fact, by theorem (\ref{thmseparate}) it is sufficient to prove that $\mathcal {CP}$ is compact. But this compactness is stated in theorem (\ref{CP-compact}). \\
Let us make clear why a separate space is not trivial:
\begin{description}
	\item[a] If $\mathcal {CP}/\sim$ is reduced to a singleton then all copulas are equivalent, which is a contradiction.
	\item[b]If $\mathcal{CP}/\sim=\{\dot{C_1},\dot{C_2}\}$ with $C_1\ne C_2$ then its topology may not be separate unless it is discrete. But in this case $\mathcal{CP}$ is not connected (thus not convex) space since $\mathcal{CP}=P^{-1}\{C_1\}\cup P^{-1}\{C_1\} $ which splits $\mathcal{CP}$ on two disjoint open sets. This contradicts theorem (\ref{CP-compact}).
\end{description}
Finally, the set $\mathcal {CP}$ contains more than three elements. This result yields a beforehand possibility to construct at least three  copulas with different measure values according to the order $\preceq.$ In other words, the order proposed by Sigurov et.al in \cite{Si16} is consistent.

\end{proof}

\section{Generalization of asymmetry measure of copulas }
 In \cite{Ne06}, a subcopula is defined as a generalization of copula. It is precisely a kind of " copula with domain." To make clearer the concept we recap the definition as written therein (adapted to bivariate case)
 \begin{definition}
A 2-dimensional subcopula (or $2$-subcopula, or bivariate one) is a function $C$ which satisfies

\begin{enumerate}
	\item The domain $Dom(C)$ of $C$ is such as $Dom(C)= S_1\times S_2$, where each $S_k$ is a measurable subset of $I$ containing 0 and 1.
	\item $C$ is grounded and $2$-increasing,  which means that $C$ satisfies the property (3) in definition (\ref{def1}).
	\item $C$ has (one-dimensional) margins $C_1$ and $C_2$  which satisfy
	$C_k (t) = t$ for all $t \in S_k.$ 
\end{enumerate}

 \end{definition}
The regularity of a given subcopula is warranted only  on the domain $S_1\times S_2$ where it is, thanks to Theorem \ref{lipsh_cont}, Lipschitz-continuous. On the complement of $S_1\times S_2$, its behavior may be pathologic.(See shuffle of M) It yields that subcopulas deserve another treatment concerning asymmetry measure. Here we will recall and then adapt the  $\mu_p-$measures defined in \cite{Si16} to the context of subcopulas. Our approach is based on Theorem \ref{Cal-Zyg} and Proposition \ref{Propos-Zygm} of section 2. \\
For a given subcopula $C$  we consider its \emph{bracket}  defined by  $C_s=|C-C^{\top}|.$ The behavior of $C_s$ may be arbitrarily impredictable on $I^2.$ So we split it in accordance with Proposition \ref{Propos-Zygm}:  $C_s=g+b$.\\
In order to compare the asymmetry of two subcopulas $C_1$ and $C_2$ and since the respective bad parts $b_1$ and $b_2$ of brackets $C_{s,1}$ and  $C_{s,2}$ are  \emph{in average} null (equal zero), it will be more convenient to compare good parts $g_1$ and $g_2.$ This yields that the following definition makes sense

\begin{definition}
 Let us fix $t\in ]0,1[$ and  let $C_1$ and $C_2$ two subcopulas with brackets $C_{s,2}$ and $C_{s,1}$ We will say that $C_1$ is more symmetric than $C_2$ with tolerance $t$ and write $C_2\prec_t C_1$ if and only if $\|g_1\|_1\ge \|g_2\|_1$, where $g_1$ and $g_2$ are the good parts associated  with  $C_1$ and $C_2$ respectively as given in proposition (\ref{Propos-Zygm}). 

\end{definition}\label{defnewmes}
For instance, we consider the measure $\mu_1$ which quantifies $\mu_1(f)=\mu(f)=\|g\|_1$ for any $f\in L^1(I^2)$ with good part $g.$ One may follow the same way for generalizing  to $p-$measures $\|g\|_p.$ \\
The parameter $t$ or \emph{tolerance} which appears in Definition (\ref{defnewmes}) is an a priori asymmetry accepted (ou tolerated) degree. For convenience and as we are dealing with copulas, it is natural to choose $t=1$ since a copula is joint distribution function.\\
To accredit definition (\ref{defnewmes}), we illustrate it with the following example which gives simultaneously an application and  illustration of above purposes

\begin{example}
We propose  asymmetric Cobb-Douglas model of cardinal utility. Let us recall essential notions about this important model in economical studies. {\it Cobb-Douglas functions} are used for both production functions 
$$
Q(K,L)=K^{\alpha}L^{\beta},\ \alpha +\beta =1
$$
where $Q$ is the production quantity, $K$ denotes the capital and $L$ is the labor. The coefficients $\alpha$ and $\beta$ explain the contribution of each production factor.\\
The same functional form is also used for the utility function of two different goods $X$ and $Y.$  We often define it as
$$
U(X,Y)=X^{\alpha}Y^{1-\alpha}.
$$
Note that the hypothesis $\alpha +\beta =1$ is not a restriction since it is assumed just to simplify the computation of some characteristics such as the \emph{slope} or \emph{the marginal  rate of substitution} given by $\frac{dY}{dX}=\frac{-\beta Y}{(1-\beta)X}.$ So one may consider the more general  product $U=X^{\alpha}Y^{\beta}$ without any other conditions on coefficients $\alpha$ and $\beta$ except their positivity.\\
Let $\Q= (q_n)_{n\ge 1}$ and consider \emph{Cobb-Douglas subcopula} $C$ defined by 

$$
C(u,v)=\frac{2}{3}u^{\alpha}v\mathbf{1}_{\overline{\Q^2}}+ \frac{1}{3}q_nq_m\mathbf{1}_{\Q^2}(q_n,q_m),\  0<\alpha<1.
$$

It is easy to check that $C$ is not symmetric and that the asymmetric part is $\frac{2}{3}u^{\alpha}v\mathbf{1}_{\overline{\Q^2}}.$ So, the bracket  $C_s$  is reduced to $C_s=\vert \frac{2}{3}u^{\alpha}v\mathbf{1}_{\overline{\Q^2}}-\frac{2}{3}v^{\alpha}u\mathbf{1}_{\overline{\Q^2}}\vert$ which will measure asymmetry degree of the copula $C.$ In other words, the bad part in $C_s$ bracket is null.

Let us know consider the copula
$$
D(u,v)=\frac{2}{3}uv\mathbf{1}_{\overline{\Q^2}}+ \frac{1}{3}q_n^{\alpha}q_m\mathbf{1}_{\Q^2}(q_n,q_m),\  0<\alpha<1.
$$

The copula $D$ is clearly non symmetric. In order to be convinced, it suffices to see example 2.11 in \cite{Ge11} where it was shown that $D$ and $C$ suggested in both of two last examples are asymmetric copulas. To measure the asymmetry of $D$, one may first  compute its bracket:
$$
D_s(u,v)=\vert \frac{1}{3}q_n^{\alpha}q_m-\frac{1}{3}q_nq_m^{\alpha}\vert \mathbf{1}_{\Q}(u,v) \ \mbox{when } \ u=q_n\  \mbox{and } \ v=q_m.
$$
\end{example}

Let us discuss more copulas $C$ and $D.$ Since they are not symmetric, the two inequalities $K \preceq C$ and $K \preceq D$ hold for any symmetric copula $K$ because the bracket of  $K$ satisfies $K_s=|K-K^{\top}|=0.$ It is easy to check that $C$ and $D$ are not comparable with respect to preoder $\preceq$. On the other hand one may use the measure $\mu$ to see that $\mu(C)>\mu(D)$ since $D_s=0 $ almost everywhere. The measure $\mu$ gives more than the simple ranking, it allows a comparison of magnitudes which means the contribution margin of each production factor ($K$ and $L$) or in general for each good ($X$ and $Y$).\\

\section*{Comments}  
\begin{enumerate}
\item[1.]
Further steps to modeling more complicated utility functions which vary with respect the time will be investigated in our future work. The idea consists on gluing such functions following the same way of frozen coefficients as in \cite{Sa15}.
\item[2.] Using any scientific computing software, one may compute the measure $\mu(h)$ of a  given copula $h: (u,v)\mapsto k u^{\alpha}v , \ k,\alpha \in ]0,1[. $ It suffices to implement the following (e.g with Pyton):\\
from scipy.integrate import dblquad\\
def Integrale(alpha,k):
 area=(dblquad(k u,v: u**alpha*v, 0, 1, k u:0, k u=1))\\
 return(k*area[0])\\
 I=Integrale(1,1)\\
print(I)
\end{enumerate}

\end{document}